\documentclass[%
  paper=a4,%
  11pt,%
  oneside,%
  DIV12,%
]{scrartcl}
\usepackage[T1]{fontenc}
\usepackage[utf8]{inputenc}
\usepackage{textcomp}
\usepackage{tgpagella} 
\usepackage{eulervm}
\usepackage{graphicx}
\usepackage{overpic}
\usepackage[outline]{contour}
\contourlength{1.0pt}
\usepackage{amsmath,amsfonts,amssymb,amsthm}
\usepackage{microtype}
\usepackage{mathtools}
\usepackage{enumitem}
\setlist{noitemsep,leftmargin=*}
\usepackage{titlesec}
\usepackage[%
  sort,%
  compress,%
  numbers,%
]{natbib}

\setkomafont{title}{\scshape}

\titleformat{\section}
  {\large\scshape\centering}
  {\thesection.}{1em}{}
  \titlespacing*{\section}
  {0ex}{3.5ex plus .1ex minus .2ex}{1.5ex minus .1ex}

\titleformat{\subsection}[runin]
  {\normalfont\itshape}
  {\S~\thesubsection.}{\wordsep}{}[.]
  \titlespacing{\subsection}
  {0pt}{1ex}{\wordsep}
\titleformat{\subsubsection}[runin]
  {\normalfont}
  {\thesubsubsection.}{\wordsep}{}[.]
  \titlespacing{\subsubsection}
  {0pt}{1ex}{\wordsep}

\newtheoremstyle{hpstheorem}%
  {3pt}
  {3pt}
  {\itshape}
  {}
  {\scshape}
  {.}
  {.5em}
  {}
\newtheoremstyle{hpsdefinition}%
  {3pt}
  {3pt}
  {\upshape}
  {}
  {\scshape}
  {.}
  {.5em}
  {}
\theoremstyle{hpstheorem}
\newtheorem{theorem}{Theorem}
\newtheorem{proposition}[theorem]{Proposition}
\newtheorem{corollary}[theorem]{Corollary}
\newtheorem{lemma}[theorem]{Lemma}
\theoremstyle{hpsdefinition}
\newtheorem{definition}[theorem]{Definition}

\usepackage{hyperref}
\usepackage{bookmark}

\hypersetup{
  pdftitle={Discrete Laplace Cycles of Period Four},%
  pdfauthor={Hans-Peter Schröcker},%
  pdfkeywords={discrete conjugate net, Laplace transform, asymptotic transform, discrete W-congruence, discrete projective differential geometry},%
  pdfproducer={},%
  pdfcreator={},%
  pdfstartview={FitH},%
  pdfborder={0 0 0},%
  colorlinks=false,%
}

\newcommand*{\KSet}{\mathbb{K}}
\newcommand*{\ZSet}{\mathbb{Z}}
\newcommand*{\Lspace}[1][3]{\mathbb{L}^{#1}}
\newcommand*{\Pspace}[1][3]{\mathbb{P}^{#1}}
\newcommand*{\Laplace}[1]{\mathcal{L}_{#1}}
\newcommand*{\Oplane}[1]{\mathcal{O}_{#1}}
\newcommand*{\Tplane}{\mathcal{T}}
\newcommand*{\CR}[4]{\mathop{CR}(#1,#2;#3,#4)}
\newcommand*{\quadric}{\mathcal{Q}}
\newcommand*{\conic}{C}
\newcommand*{\projection}[3]{#1{\raisebox{2pt}{\ensuremath{\uparrow}}}_{#2}^{#3}}

\title{Discrete Laplace Cycles of Period Four}
\author{Hans-Peter Schröcker\\
Unit Geometry and CAD, University Innsbruck}

\begin{document}

\maketitle

\begin{abstract}
  We study discrete conjugate nets whose Laplace sequence is of period
  four. Corresponding points of opposite nets in this cyclic sequence
  have equal osculating planes in different net directions, that is,
  they correspond in an asymptotic transformation. We show that this
  implies that the connecting lines of corresponding points form a
  discrete W-congruence. We derive some properties of discrete Laplace
  cycles of period four and describe two explicit methods for their
  construction.
\end{abstract}


\section{Introduction}
\label{sec:introduction}

In \cite{jonas37:_laplacesche_zyklen}, H.~Jonas established a couple
of results pertaining to asymptotic transforms, W-congruences,
conjugate nets and Laplace cycles of period four. This article
explores discrete versions of Jonas' theorems. It blends in with and
extends more recent research on discrete conjugate nets, their Laplace
transforms, discrete W-congruences and, in the limiting case, also
discrete asymptotic nets. The original references to the discrete
theory are \cite{doliwa01:_a-nets_integrable_discrete_geometry,%
  doliwa01:_asymptotic_nets-w-congruences,%
  doliwa00:_transformations,%
  doliwa97:_toda_system,%
  nieszporski02:_discretization_asymptotic_nets}, an exposition of the
current state of research can be found in
\cite{bobenko08:_discrete_differential_geometry}.

Our contribution belongs to the field of discrete projective
differential geometry. But we try to avoid the typical differential or
difference equations and instead resort to synthetic reasoning. It is
in the nature of our subject that we merely prove some incidence
geometric results. This may seem like an old fashioned approach to an
outdated topic. It is needless to say that we do not share this
opinion, but it seems necessary to support our point of view by a few
arguments:
\begin{itemize}
\item Arguably, the lack of ordering principles is one of the reasons
  why pure projective geometry is of little interest to today's
  mathematicians. In contrast, a highly active research field,
  discrete differential geometry, has emerged from the desire to find
  discrete versions of differential geometric theorems
  \cite{bobenko08:_discrete_differential_geometry}. We feel that
  differential geometry can serve as a guideline to single out
  ``interesting'' results from the wealth of projective incidence
  theorems.
\item Our contributions generalize recent results on discrete
  conjugate nets and discrete asymptotic nets
  \cite{doliwa01:_a-nets_integrable_discrete_geometry,%
    doliwa01:_asymptotic_nets-w-congruences,%
    doliwa00:_transformations,%
    doliwa97:_toda_system,%
    nieszporski02:_discretization_asymptotic_nets}. Thus, our topic
  blends in with current research and should not be considered
  outdated.
\item Relying on synthetic reasoning is of a particular appeal in the
  absence of metric structures. The reason is, that synthetic
  arguments clearly exhibit the few fundamental assumptions on which
  the theory is based, thus producing results that can be attributed
  to a very general type of projective geometry.
\end{itemize}

Referring to the last point, our results hold true for
three-dimensional projective geometries over fields of characteristic
$\neq 2$ and with sufficiently many elements (so that the mere
formulation of certain results make sense). By Theorem~7 in
\cite[Section~50]{veblen16:_projective_geometry}, commutativity of the
underlying algebraic structure is equivalent to the validity of the
fundamental theorem of projective geometry and the existence of doubly
ruled surfaces \cite[Section~103]{veblen16:_projective_geometry}. The
last point is crucial for our considerations.

We continue this article by introducing some basic concepts and
notations, like discrete conjugate nets and their Laplace transforms,
in \autoref{sec:preliminaries}. The subsequent
\autoref{sec:asymptotic-transforms} features asymptotic transforms of
(not necessarily conjugate) nets. The main result states that the
discrete line congruence obtained by connecting corresponding points
is a discrete W-congruence, see
\cite{doliwa01:_a-nets_integrable_discrete_geometry,
  doliwa01:_asymptotic_nets-w-congruences}. Theorem~\ref{th:13} shows
how to construct asymptotic transforms on a given W-congruence. The
results of this section comprise, as a limiting case, the theory of
discrete asymptotic nets and their W-transform as laid out in
\cite{doliwa01:_a-nets_integrable_discrete_geometry,%
  doliwa01:_asymptotic_nets-w-congruences,%
  nieszporski02:_discretization_asymptotic_nets}.  Examples of
asymptotically related nets can be obtained from a Laplace sequence of
period four. They are studied in
\autoref{sec:periodic-laplace-cycles}. Opposite nets of the cycle are
asymptotically related so that their diagonal congruences are
W-congruences. We conclude this article with two different methods for
the construction of discrete Laplace cycles of period four.

Before continuing, we should mention that it is not difficult to
obtain general conditions on conjugate nets that admit periodic
Laplace cycles (in both, the smooth and the discrete case). Older
references for the smooth case are
\cite{hammond21:_periodic_conjugate_nets,%
  barner58:_geschlossene_laplace_ketten,%
  degen60:_geschlossene_ungerade_laplace_ketten}, newer contributions
include \cite{hu01:_laplace_sequences_toda_equations} and
\cite[Section~4.4]{sasaki06:_line_congruence_transformation}.  The
main problem is not the derivation of conditions but their geometric
interpretation. The mentioned references provide some general results
in this direction for the smooth case. This article is among the first
contributions in the discrete setting.

\section{Preliminaries}
\label{sec:preliminaries}

We denote by $\ZSet$ the set of integers and by $\Pspace$ the
projective space of dimension three over the field $\KSet$. We assume
that the characteristic of $\KSet$ is different from two.

\begin{definition}
  A \emph{discrete net} is a map
  \begin{equation*}
    f\colon \ZSet^2 \to \Pspace,
    \quad
    (i,j) \mapsto f(i,j) \eqqcolon f_i^j.
  \end{equation*}
  (Since we only deal with two-dimensional nets, we use the terser
  notation with upper and lower indices. The even short shift-notation
  of \cite{bobenko08:_discrete_differential_geometry} seems not
  necessary for two-dimensional nets.) A \emph{discrete conjugate net}
  is a discrete net such that every elementary quadrilateral $f_i^j$,
  $f_{i+1}^j$, $f_{i+1}^{j+1}$, $f_i^{j+1}$ is planar.
\end{definition}

The choice of $\ZSet^2$ as parameter space is only a matter of
convenience. It allows us to ignore boundary conditions. We might as
well admit sufficiently large sets of the shape
\begin{equation*}
  \{(i,j) \in \ZSet^2 \mid i_0 < i < i_1 \text{ and } j_0 < j < j_1\}
\end{equation*}
with integers $i_0 < i_1$ and $j_0 < j_1$.

Throughout this paper, we make some regularity assumptions on the nets
under consideration. Denote the span of projective subspaces by the
symbol ``$\vee$''. We call a discrete net regular if for all $(i,j)
\in \ZSet^2$ the \emph{osculating planes}
\begin{equation*}
  \Oplane{1} f_i^j \coloneqq f_{i-1}^j \vee f_i^j \vee f_{i+1}^j,
  \quad\text{and}\quad
  \Oplane{2} f_i^j \coloneqq f_i^{j-1} \vee f_i^j \vee f_i^{j+1},
\end{equation*}
in the first and second net direction are well-defined, that is, of
projective dimension two. For the regularity of a discrete conjugate
net we require additionally, that the four points
\begin{equation*}
  f_i^j,\ f_{i+1}^j,\
  f_{i+1}^{j+1},\
  f_i^{j+1}
\end{equation*}
are not collinear. Non-regular nets are called \emph{singular.}

From a discrete conjugate net $f$ we derive two further discrete
conjugate nets by means of the Laplace transform
(\cite{doliwa97:_toda_system,doliwa00:_transformations}):

\begin{definition}
  The \emph{first and second Laplace transforms} of a regular discrete
  conjugate net are the discrete conjugate nets
  \begin{equation*}
    \begin{aligned}
      \Laplace{1} f\colon \ZSet^2 \to \Pspace,\quad
      (i,j) \mapsto (f_i^j \vee f_{i+1}^j) \cap (f_i^{j+1} \vee f_{i+1}^{j+1}),\\
      \Laplace{2} f\colon \ZSet^2 \to \Pspace,\quad
      (i,j) \mapsto (f_i^j \vee f_i^{j+1}) \cap (f_{i+1}^j \vee f_{i+1}^{j+1}).
    \end{aligned}
  \end{equation*}
  By regularity of $f$, they are well-defined but not necessarily
  regular.
\end{definition}

Note that we adopt the indexing convention of
\cite[pp.~76--77]{bobenko08:_discrete_differential_geometry} since it
is more convenient for our purposes than the convention used in of
\cite{doliwa97:_toda_system,doliwa00:_transformations}. An example is
depicted in the lower right drawing of
\autoref{fig:laplace-construction}, where $\Laplace{1}f \coloneqq h$
and $\Laplace{2}f \coloneqq h$. It is easy to see that $\Laplace{1} f$
and $\Laplace{2} f$ are indeed discrete conjugate nets due to
\begin{equation*}
  \{ \Laplace{k} f_i^j,\ \Laplace{k} f_{i+1}^j,\ \Laplace{k} f_i^{j+1},\ \Laplace{k} f_{i+1}^{j+1} \} \subset \Oplane{k}f_{i+1}^{j+1},
  \quad k \in \{1, 2\}.
\end{equation*}

Let us study repeated application of the Laplace transform. Because of
$\Laplace{1}\Laplace{2}f_i^j = \Laplace{2}\Laplace{1}f =
f_{i+1}^{j+1}$, the composition of two Laplace transforms in different
net directions effects only an index shift of $f$. This is not
interesting so that we focus on the compositions of Laplace transforms
in the same net direction:

\begin{definition}
  The \emph{$i$-th Laplace sequence} to a discrete conjugate net $f$
  is the sequence $l \mapsto \Laplace{i}^l f$ where $\Laplace{i}^lf$
  is recursively defined by $\Laplace{i}^lf = \Laplace{i}
  \Laplace{i}^{l-1}f$ and $\Laplace{i}^{0} f = f$.
\end{definition}

In general, both Laplace sequences
\begin{equation*}
  \Laplace{1}^0f, \Laplace{1}^1f, \Laplace{1}^2f, \ldots
  \quad\text{and}\quad
  \Laplace{2}^0f, \Laplace{2}^1f, \Laplace{2}^2f, \ldots
\end{equation*}
are infinite, at least if the field $\KSet$ is infinite. It is,
however, possible that the iterated construction breaks down at some
point due to net singularities. Another possibility is that the
sequences become periodic. An instance of this is precisely the case
we are interested in:

\begin{definition}
  The Laplace sequences of a discrete conjugate net $f$ are called a 
  \emph{Laplace cycle of period four}, if $\Laplace{11} f =
  \Laplace{22} f$.
\end{definition}

Indeed, in this case we have $\Laplace{1}^{l+4} = \Laplace{1}^l$ and
$\Laplace{2}^{l+4} = \Laplace{2}^l$ for any positive integer $l$. We
suggest to think of a Laplace cycle of period four in terms of the
cyclic sequence of four conjugate nets
\begin{equation*}
  f,\quad
  h \coloneqq \Laplace{1} f,\quad
  g \coloneqq \Laplace{11} f = \Laplace{22} f,\quad
  k \coloneqq \Laplace{2} f.
\end{equation*}
The first or second Laplace sequence started from any of the nets $f$,
$h$, $g$, and $k$ are identical, up to indexing. Thus, the four nets
can be treated on equal footing.  We call the nets $f$ and $g$ (as
well as $h$ and $k$) \emph{opposite} nets of the cycle. Examples are
depicted in \autoref{fig:laplace-construction} and
\autoref{fig:laplace-construction-3d}.

\section{Asymptotic transforms and W-congruences}
\label{sec:asymptotic-transforms}

In a Laplace cycle of period four, the osculating planes in different
directions of opposite nets coincide, for example $\Oplane{1} f =
\Oplane{2} g$ and $\Oplane{2} f = \Oplane{1} g$. Assume that
corresponding points do not coincide ($f_i^j \neq g_i^j$, $h_i^j \neq
k_i^j$) and that, at every vertex, the two osculating planes are
different ($\Oplane{1} f_i^j \neq \Oplane{2}f_i^j$, $\Oplane{1} h_i^j
\neq \Oplane{2}h_i^j$, etc.). Setting $K_i^j \coloneqq f_i^j \vee
g_i^j$ and $L_i^j \coloneqq h_i^j \vee k_i^j$, we then have
\begin{alignat*}{3}
  K_i^j & = \Oplane{1} f_i^j \cap \Oplane{2} f_i^j &  & = \Oplane{2} g_i^j \cap \Oplane{1} g_i^j, \\
  L_i^j & = \Oplane{1} h_i^j \cap \Oplane{2} h_i^j &  & = \Oplane{2} k_i^j \cap \Oplane{1} k_i^j.
\end{alignat*}
In order to capture this relation between $f$ and $g$ (or $h$ and $k$)
in more generality, we study the axis congruence of a discrete net and
its asymptotic transforms. Our terminology goes back to the smooth
case and in particular to \cite{jonas37:_laplacesche_zyklen} and
\cite{wilczynski15:_general_theory_congruences}.

\begin{definition}
  \label{def:axis-congruence}
  The \emph{axis congruence} of a (not necessarily conjugate) net
  $f\colon \ZSet^2 \to \Pspace$ is the map that sends a point $(i,j)
  \in \ZSet^2$ to the line $\Oplane{1} f_i^j \cap \Oplane{2}
  f_i^j$. It is well-defined only for vertices where the two
  osculating planes are different.
\end{definition}

\begin{definition}
  \label{def:asymptotic-transform}
  Two discrete nets $f$ and $g$ are called \emph{asymptotically
    related} or \emph{asymptotic transforms} of each other, if
  $\Oplane{1} f = \Oplane{2} g$ and $\Oplane{2} f = \Oplane{1} g$.
\end{definition}

Using these definitions, we may rephrase our findings of
\autoref{sec:preliminaries} as

\begin{proposition}
  \label{prop:7}
  Opposite nets in a discrete Laplace cycle of period four have the
  same axis congruence and are asymptotically related.
\end{proposition}

Note that the concepts of osculating plane, axis congruence and
asymptotic transform make perfect sense in the smooth setting and
Proposition~\ref{prop:7} holds true as well,
see~\cite{jonas37:_laplacesche_zyklen}.

For an important class of discrete nets the axis congruence is
undefined at every vertex. These nets are characterized by $\Oplane{1}
f = \Oplane{2} f$ and are called \emph{asymptotic nets} or
\emph{A-nets}, see \cite{doliwa01:_asymptotic_nets-w-congruences,%
  doliwa01:_a-nets_integrable_discrete_geometry,%
  nieszporski02:_discretization_asymptotic_nets} or
\cite[Section~2.4]{bobenko08:_discrete_differential_geometry}. We
usually exclude them from our considerations. It should, however, be
mentioned that asymptotic nets can be obtained from a general net $f$
by a suitable passage to the limit that satisfies $\lim\Oplane{1}f =
\lim\Oplane{2}f$. In this case, the common limit of both osculating
planes is the tangent plane $\Tplane f$ at the respective vertex. Two
asymptotic nets $f_i^j$ and $g_i^j$ are said to be \emph{W-transforms}
of each other, if the lines $f_i^j \vee g_i^j$ are contained in both
tangent planes $\Tplane f_i^j$ and $\Tplane g_i^j$. Pairs of
W-transforms appear as limit of pairs of asymptotic
transforms. Indeed, our theory comprises the theory of discrete
asymptotic nets and their W-transforms as limiting case. In
\autoref{sec:asymptotic-nets} we explicitly describe how to set up
this limiting process.

We denote the Grassmannian of lines in $\Pspace$ by $\Lspace$.  Via
the Klein map, a straight line is identified with a point on a quadric
in $\Pspace[5]$, the \emph{Plücker quadric.} Klein map and Plücker
quadric in real projective three-space are described in detail in
\cite[Section~2.1]{pottmann10:_line_geometry}. There is no essential
difference in projective spaces over commutative fields.

\begin{definition}
  A map $L\colon \ZSet^2 \to \Lspace$ is called a \emph{W-congruence}
  if its Klein image on the Plücker quadric is a conjugate net.
\end{definition}

W-congruences over real projective spaces have been introduced in
\cite{doliwa01:_a-nets_integrable_discrete_geometry} as discrete line
congruences that connect corresponding points of a discrete asymptotic
net and its W-transform. The necessity of our defining property is
also established their. Its sufficiency is not difficult to see and,
in fact, is implicitly used in later publications. A more elementary
characterization of W-congruences simply demands that the four lines
$A_i^j$, $A_{i+1}^j$, $A_{i+1}^{j+1}$, $A_i^{j+1}$ belong to a regulus
\cite[Section~103]{veblen16:_projective_geometry}, that is, they are
skew lines on a doubly ruled surface.

\subsection{W-congruences as axis congruence of asymptotic transforms}
\label{sec:W-congruence-from-asymptotic-transforms}

Our main result on asymptotic transforms is the discrete version of
\cite[p.~248]{jonas37:_laplacesche_zyklen}.

\begin{theorem}
  \label{th:9}
  If $f$ and $g$ are asymptotic transforms of each other, their common
  axis congruence is a W-congruence.
\end{theorem}

\begin{proof}
  Denote the lines of the common axis congruence of $f$ and $g$ by
  $A_i^j = f_i^j \vee g_i^j$. Because of $\Oplane{1} f_{i+1}^j =
  \Oplane{2} g_{i+1}^j$ and $\Oplane{2} f_i^{j+1} = \Oplane{1}
  g_i^{j+1}$ the straight line $f_i^j \vee g_{i+1}^{j+1}$ intersects
  the four lines
  \begin{equation}
    \label{eq:1}
    A_i^j,\
    A_{i+1}^j,\
    A_{i+1}^{j+1},
    \quad\text{and}\quad
    A_i^{j+1}.
  \end{equation}
  Similar reasoning shows that the same is true for the lines
  $f_{i+1}^j \vee g_i^{j+1}$, $f_{i+1}^{j+1} \vee g_i^j$, and
  $f_i^{j+1} \vee g_{i+1}^j$. Thus, the four lines \eqref{eq:1} admit
  four transversal lines. This is only possible, if they are skew
  generators on a hyperboloid.
\end{proof}

The proof of Theorem~\ref{th:9} allows a different interpretation. The
point $g_{i+1}^{j+1}$ is the projection of $f_i^j$ onto $A_{i+1,j+1}$
from the center $A_{i+1}^j$ and, at the same time, from the center
$A_i^{j+1}$. The W-congruence property guarantees that both
projections always yield the same result. This projective relation
between the ranges of points $f_i^j \in A_i^j$ and $g_{i+1}^{j+1} \in
A_{i+1}^{j+1}$ is important enough to introduce a new notation.

\begin{definition}
  Given three pairwise skew lines $A$, $B$, $C$ we denote by
  $\projection{C}{A}{B}$ the projection of $A$ onto $B$ from the
  center $C$. More precisely, $\projection{C}{A}{B}$ maps the point $a
  \in A$ to the point $(a \vee C) \cap B$.
\end{definition}

With this notation, we have for example
\begin{equation*}
  \projection{A_i^j}{A_{i-1}^j}{A_i^{j+1}} =
  \projection{A_{i-1}^{j+1}}{A_{i-1}^j}{A_i^{j+1}}
  \quad\text{or}\quad
  \projection{A_i^j}{A_i^{j+1}}{A_{i+1}^j} =
  \projection{A_{i+1}^{j+1}}{A_i^{j+1}}{A_{i+1}^j}.
\end{equation*}
When range and image line stem from the same line congruence, we make
this notation a little more readable by writing only the upper and
lower index:
\begin{equation}
  \label{eq:2}
  \projection{A_i^j}{i-1,j}{i,j+1} =
  \projection{A_{i-1}^{j+1}}{i-1,j}{i,j+1}
  \quad\text{or}\quad
  \projection{A_i^j}{i,j+1}{i+1,j} =
  \projection{A_{i+1}^{j+1}}{i,j+1}{i+1,j}.
\end{equation}

A diagram of diverse projections between the lines of the congruence
$A$ is depicted in \autoref{fig:diagram}. Equation~\eqref{eq:2}
justifies the drawing of diagonal arrows without indicating the
projection center. For example, the arrow between $A_i^j$ and
$A_{i+1}^{j+1}$ denotes projection between these two lines from either
$A_{i+1}^j$ or~$A_i^{j+1}$.

\begin{figure}
  \centering
  \includegraphics{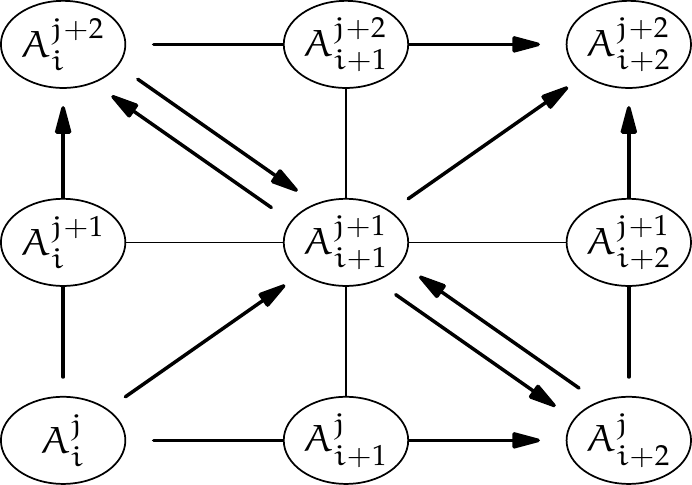}
  \caption{Diagram of projections}
  \label{fig:diagram}
\end{figure}

\subsection{Asymptotic transforms on a given W-congruence}
\label{sec:asymptotic-transforms-on-W-congruence}

As a natural next step we investigate the question whether it is
possible to construct a pair $f$, $g$ of asymptotically related nets
whose vertices lie on a given W-congruence $A$. This turns out to be
feasible in multiple ways. Essentially, we are allowed to choose the
values of $f$ (or $g$) on an elementary
quadrilateral. Theorem~\ref{th:13} below is a little more general. Its
proof requires an auxiliary result.

\begin{lemma}
  \label{lem:11}
  $\projection{A_i^j}{i-1,j}{i+1,j} =
  \bigl(\projection{A_i^j}{i,j+1}{i+1,j}\bigr) \circ
  \bigl(\projection{A_i^j}{i-1,j}{i,j+1}\bigr)$ and
  $\projection{A_i^j}{i,j-1}{i,j+1} =
  \bigl(\projection{A_i^j}{i+1,j}{i,j+1}\bigr) \circ
  \bigl(\projection{A_i^j}{i,j-1}{i+1,j}\bigr)$.
\end{lemma}

\begin{proof}
  Consider a point $x \in A_i^j$.  It defines the plane $\xi = x \vee
  A_{i+1}^j$. Clearly, $y \coloneqq
  \projection{A_{i+1}^j}{i,j}{i+1,j+1}(x) = A_{i+1}^{j+1} \cap \xi$
  and thus also $\projection{A_{i+1}^j}{i+1,j+1}{i+2,j}(y) = A_{i+2}^j
  \cap \xi = \projection{A_{i+1}^j}{i,j}{i+2,j}(x)$. The second claim
  is obtained from the first by interchanging the two net directions.
\end{proof}

The interpretation of Lemma~\ref{lem:11} in the diagram of
\autoref{fig:diagram} is that the projection paths along the sides of
the rectangle can be decomposed into the paths from start point to
center and from center to end-point.

\begin{definition}
  We call an index pair $(i,j)$ \emph{even-even} if both $i$ and $j$
  are even. Similarly, we speak of \emph{odd-odd, even-odd, or
    odd-even} index pairs. We say that two index pairs are of the same
  \emph{parity} if both are even-even, even-odd, odd-even, or
  odd-odd. A \emph{black vertex} is even-even or odd-odd and a
  \emph{white vertex} is either even-odd or odd-even.
\end{definition}

\begin{theorem}
  \label{th:13}
  On a given W-congruence we can find a four-parametric set of pairs
  of asymptotic transforms. Every such pair $f$, $g$ is uniquely
  determined by the values of $f$ at four vertices of pairwise
  different parity.
\end{theorem}

\begin{proof}
  Assume that the vertex $f_i^j$ is given. Clearly, we have
  \begin{equation*}
    \begin{aligned}
      f_{i+2}^j &= \projection{A_{i+1}^j}{i,j}{i+2,j}(f_i^j), &
      f_{i-2}^j &= \projection{A_{i-1}^j}{i,j}{i-2,j}(f_i^j),\\
      f_i^{j+2} &= \projection{A_i^{j+1}}{i,j}{i,j+2}(f_i^j), &
      f_i^{j-2} &= \projection{A_i^{j-1}}{i,j}{i,j-2}(f_i^j).
    \end{aligned}
  \end{equation*}
  Proceeding in like manner, we can construct the values of $f$ on all
  vertices of the same parity as $(i,j)$. Thus, if there exists a net
  $f$ with axis congruence $A$, it is uniquely determined by the
  initial data. Since the net $f$ also determines all osculating
  planes $\Oplane{1} f_i^j$, $\Oplane{2} f_i^j$, the net $g$ is
  determined as well, for example $g_{i+1}^j = A_{i+1}^j \cap
  \Oplane{1}f_i^j$. Because of $\projection{A_{i+1}^j}{i,j}{i+1,j+1} =
  \projection{A_i^{j+1}}{i,j}{i+1,j+1}$, the construction of $g$ is
  not ambiguous. Moreover, $f$ and $g$ are asymptotic transforms with
  axis congruence~$A$.

  We still have to show existence of $f$. This is necessary because
  our construction of the vertices of $f$ might produce a
  contradiction. The prototype case of this is the construction of
  $f_{i+2,j+2}$ in two ways according to
  \begin{equation}
    \label{eq:3}
    f_i^j \xrightarrow{\projection{A_{i+1}^j}{i,j}{i+2,j}} f_{i+2}^j \xrightarrow{\projection{A_{i+2}^{j+1}}{i+2,j}{i+2,j+2}} f_{i+2}^{j+2},
    \quad
    f_i^j \xrightarrow{\projection{A_i^{j+1}}{i,j}{i,j+2}} f_i^{j+2} \xrightarrow{\projection{A_{i+1}^{j+2}}{i,j+2}{i+2,j+2}} f_{i+2}^{j+2}.
  \end{equation}
  If we can show that both routes in \eqref{eq:3} yield the same
  value, the validity of our construction is guaranteed. In other
  words, we have to show
  \begin{equation*}
    \bigl(\projection{A_{i+2}^{j+1}}{i+2,j}{i+2,j+2}\bigr) \circ \bigl(\projection{A_{i+1}^j}{i,j}{i+2,j}\bigr) =
    \bigl(\projection{A_{i+1}^{j+2}}{i,j+2}{i+2,j+2}\bigr) \circ \bigl(\projection{A_i^{j+1}}{i,j}{i,j+2}\bigr).
  \end{equation*}
  A glance at \autoref{fig:diagram} immediately confirms this: The
  projection paths along two adjacent sides of the rectangle equals
  the diagonal projection path. A formal argument uses
  Lemma~\ref{lem:11} and the identities \eqref{eq:2}. We have
  \begin{equation}
    \label{eq:4}
    \begin{aligned}
      & \bigl(\projection{A_{i+2}^{j+1}}{i+2,j}{i+2,j+2}\bigr)
      \circ
      \bigl(\projection{A_{i+1}^j}{i,j}{i+2,j}\bigr) \\
      ={}&
      \bigl(\projection{A_{i+2}^{j+1}}{i+1,j+1}{i+2,j+2}\bigr) \circ
      \bigl(\projection{A_{i+2}^{j+1}}{i+2,j}{i+1,j+1}\bigr) \circ
      \bigl(\projection{A_{i+1}^j}{i+1,j+1}{i+2,j}\bigr) \circ
      \bigl(\projection{A_{i+1}^j}{i,j}{i+1,j+1}\bigr) \\
      ={}&
      \bigl(\projection{A_{i+2}^{j+1}}{i+1,j+1}{i+2,j+2}\bigr) \circ
      \underbrace{
          \bigl(\projection{A_{i+1}^j}{i+2,j}{i+1,j+1}\bigr) \circ
          \bigl(\projection{A_{i+1}^j}{i+1,j+1}{i+2,j}\bigr)}_{\text{identity on $A_{i+1}^{j+1}$}}
      \circ
      \bigl(\projection{A_{i+1}^j}{i,j}{i+1,j+1}\bigr) \\
      ={}&
      \bigl(\projection{A_{i+2}^{j+1}}{i+1,j+1}{i+2,j+2}\bigr) \circ
      \bigl(\projection{A_{i+1}^j}{i,j}{i+1,j+1}\bigr).
    \end{aligned}
  \end{equation}
  By interchanging the two net directions, we arrive at
  \begin{equation}
    \label{eq:5}
    \bigl(\projection{A_{i+1}^{j+2}}{i,j+2}{i+2,j+2}\bigr)
    \circ
    \bigl(\projection{A_i^{j+1}}{i,j}{i,j+2}\bigr) 
    =
    \bigl(\projection{A_{i+1}^{j+2}}{i+1,j+1}{i+2,j+2}\bigr)
    \circ
    \bigl(\projection{A_i^{j+1}}{i,j}{i+1,j+1}\bigr).
  \end{equation}
  Equation~\eqref{eq:2} shows that the expressions after the last
  equal sign in \eqref{eq:4} and \eqref{eq:5} are equal.
\end{proof}

It is interesting to compare Theorem~\ref{th:13} with the results of
\citep{jonas37:_laplacesche_zyklen}. While Theorem~\ref{th:13} shows
existence of a four-parametric set of asymptotically related nets on
the discrete W-congruence $A$, there only exists a one-parametric set
of such surfaces on a smooth W-congruence.

\subsection{Asymptotic nets as limiting case}
\label{sec:asymptotic-nets}

We already mentioned that asymptotic nets related by a W-transform can
be seen as limiting case of pairs of asymptotically related
nets. Here, we provide more details on this remark.

Consider a W-congruence $A$. By Theorem~\ref{th:13}, we can construct
pairs $f$, $g$ of asymptotic transforms on $A$ by prescribing suitable
values of $f$ on an elementary quadrilateral. The value of $f$ on an
even-even vertex determines the values of $\Oplane{1}f = \Oplane{2}g$
on all odd-even vertices, the values of $\Oplane{2}f = \Oplane{1}g$ on
all even-odd vertices, and the values of $g$ on all odd-odd vertices.
Alternatively, we can also prescribe the values of $\Oplane{1}f$ and
$\Oplane{2}f$ on a black and a white vertex. If we make the additional
requirement $\Oplane{1}f = \Oplane{2}f =\Oplane{1}g = \Oplane{2} g$, a
similar construction is possible but only one plane on a black and one
plane on a white vertex can be chosen. The resulting nets $f$ and $g$
will then form a pair of asymptotic nets, related by a W-transform.

Consider now four sequences
\begin{equation}
  \label{eq:6}
  (\Oplane{1} f_0^0)_n,\quad
  (\Oplane{2} f_0^0)_n,\quad
  (\Oplane{1} f_1^0)_n,\quad
  (\Oplane{2} f_1^0)_n
\end{equation}
of planes such that for each $n$
\begin{equation*}
  A_0^0 \subset (\Oplane{1}f_0^0)_n,\quad
  A_0^0 \subset (\Oplane{2}f_0^0)_n,\quad
  A_1^0 \subset (\Oplane{1}f_1^0)_n,\quad
  A_1^0 \subset (\Oplane{2}f_1^0)_n,
\end{equation*}
and
\begin{equation*}
  \lim_{n\to\infty} (\Oplane{1} f_0^0)_n = \lim_{n\to\infty} (\Oplane{2} f_0^0)_n,
  \quad
  \lim_{n\to\infty} (\Oplane{1} f_1^0)_n = \lim_{n\to\infty} (\Oplane{2} f_1^0)_n.
\end{equation*}
For each $n$, the planes in \eqref{eq:6} define asymptotically related
nets $f_n$ and $g_n$ and the point-wise limits $f =
\lim_{n\to\infty}f_n$ and $g=\lim_{n\to\infty}g_n$ exist. By
construction, the nets $f$ and $g$ are asymptotic nets related by a
W-transform.

\subsection{Two further results}
\label{sec:further-results}

We conclude this section with two simple but curious observations on
the cross-ratio of asymptotic transforms on the same W-congruence and
on their local configuration.

\begin{corollary}
  \label{cor:14}
  Consider two pairs $f$, $g$ and $f'$, $g'$ of asymptotically related
  conjugate nets on the same W-congruence $A$. Then the cross-ratio
  \begin{equation}
    \label{eq:7}
    \CR{f_{i,j}}{f'_{i,j}}{g_{i,j}}{g'_{i,j}}
  \end{equation}
  is constant on all black and on all white vertices, respectively.
\end{corollary}

\begin{proof}
  Because vertices of the same parity are obtained from the vertices
  on $A_{i,j}$ by a series of projections, they give rise to the same
  cross-ratio \eqref{eq:7}.  But also the vertices $g_{i+1,j+1}$ and
  $g'_{i+1,j+1}$ are the projections of $f_{i,j}$ and $f'_{i,j}$ from
  the center $A_{i+1,j}$ (or from the center $A_{i,j+1}$). Thus,
  \begin{equation*}
    \begin{aligned}
    \CR{f_{i,j}}{f'_{i,j}}{g_{i,j}}{g'_{i,j}} &=
    \CR{g_{i+1,j+1}}{g'_{i+1,j+1}}{f_{i+1,j+1}}{f'_{i+1,j+1}} \\
    & = \CR{f_{i+1,j+1}}{f'_{i+1,j+1}}{g_{i+1,j+1}}{g'_{i+1,j+1}}.
    \end{aligned}
  \end{equation*}
  This is precisely what had to be shown.
\end{proof}

In \cite{jonas37:_laplacesche_zyklen}, H.~Jonas proved that for each
line of a smooth W-congruence $A$, the vertices of two asymptotic
transforms $f$ and $g$ on $A$ are harmonic with respect to the two
asymptotically parametrized focal nets of $A$. In view of
Theorem~\ref{th:13} and \autoref{sec:asymptotic-nets}, this theorem
cannot have a discrete version and Corollary~\ref{cor:14} is the
closest result we can get.

Our next result studies the local geometry of an elementary
quadrilateral in asymptotically related nets. Recall that two
quadruples of points $(a_0,a_1,a_2,a_3)$ and $(b_0,b_1,b_2,b_3)$ form
a pair of Möbius tetrahedra if
\begin{equation}
  \label{eq:8}
  a_i \in b_j \vee b_k \vee b_l
  \quad\text{and}\quad
  b_i \in a_j \vee a_k \vee a_l
\end{equation}
for any choice of pairwise different $i$, $j$, $k$, $l \in
\{1,2,3,4\}$. In other words, the vertices of one tetrahedron lie in
the face planes of the other.

\begin{proposition}
  \label{prop:15}
  Assume that $f$ and $g$ is a pair of asymptotically related
  nets. Then the two quadruples
  \begin{equation*}
    (a_0,a_1,a_2,a_3) \coloneqq (f_{0,0}, f_{1,0}, g_{0,1}, g_{1,1})
    \quad\text{and}\quad
    (b_0,b_1,b_2,b_3) \coloneqq (f_{1,1}, f_{0,1}, g_{1,0}, g_{0,0})
  \end{equation*}
  form the vertices of a pair of Möbius tetrahedra.
\end{proposition}

\begin{proof}
  The proof is simply a matter of comparing the Möbius conditions
  \eqref{eq:8} with the conditions on the asymptotic transform. For
  example $a_0 \in b_1 \vee b_2 \vee b_3$ is true because of
  $\Oplane{2} f_{0,0} = \Oplane{1} g_{0,0}$.
\end{proof}

Proposition~\ref{prop:15} yields a different possibility for proving
Theorem~\ref{th:9} by means of
\cite[Equation~(1.0)]{edge36:_quadrics_moebius_tetrads} and a
straightforward calculation. The classic result of Möbius in
\cite{moebius28:_moebius_tetrahedra} states that seven of the eight
incidence conditions \eqref{eq:8} imply the eighth. This means that
the conditions on the asymptotic relation between two discrete nets
$f$ and $g$ are not independent either.

\section{Periodic Laplace cycles}
\label{sec:periodic-laplace-cycles}

In this section we specialize the results of
\autoref{sec:asymptotic-transforms} to opposite nets $f$, $g$ (or $h$,
$k$) in a Laplace cycle $f$, $h$, $g$, $k$ of period four.

\subsection{General results}
\label{sec:general-results}

Call the set of connecting lines of corresponding points of one pair
of opposite nets a \emph{diagonal congruence} of the cycle. As an
immediate consequence of Theorem~\ref{th:9} and the fact that opposite
nets in a Laplace cycle of period four are asymptotically related we
have

\begin{corollary}
  The two diagonal congruences of a discrete Laplace cycle of period
  four are W-congruences.
\end{corollary}

A given W-congruence $A$ is, in general, not the diagonal congruence
of a Laplace cycle of period four. In other words, it is impossible to
find conjugate nets $f$ and $g$ such that $f_i^j$, $g_i^j \in A_i^j$
for all indices $(i,j)$. But the observation at the beginning of
\cite[Section~3]{jonas37:_laplacesche_zyklen} also holds in the
discrete setting:

\begin{theorem}
  \label{th:17}
  If a W-congruence $A$ appears as the axis congruence of a discrete
  conjugate net $f$, the net $f$ gives rise to a Laplace cycle of
  period four.
\end{theorem}

This theorem is an immediate consequence of Lemma~\ref{lem:19}, below.

\begin{lemma}
  \label{lem:18}
  Consider two spatial quadrilaterals $(a_0,a_1,a_2,a_3)$ and
  $(b_0,b_1,b_2,b_3)$ such that for $i \in \{0,1,2,3\}$ the lines $a_i
  \vee a_{i+1}$ and $b_i \vee b_{i+1}$ intersect in a point $f_i$ and
  span a plane $\varphi_i$ (indices modulo four). Then the four points
  $f_0$, $f_1$, $f_2$, and $f_3$ lie in a plane $\varphi$ if and only
  if the four planes $\varphi_0$, $\varphi_1$, $\varphi_2$, and
  $\varphi_3$ contain a point $f$. In this case, the quadrilaterals
  correspond in a perspective collineation
  \cite[Section~29]{veblen16:_projective_geometry} with center $f$ and
  plane of perspectivity $\varphi$ (\autoref{fig:point-plane}).
\end{lemma}

\begin{figure}
  \centering
  \includegraphics{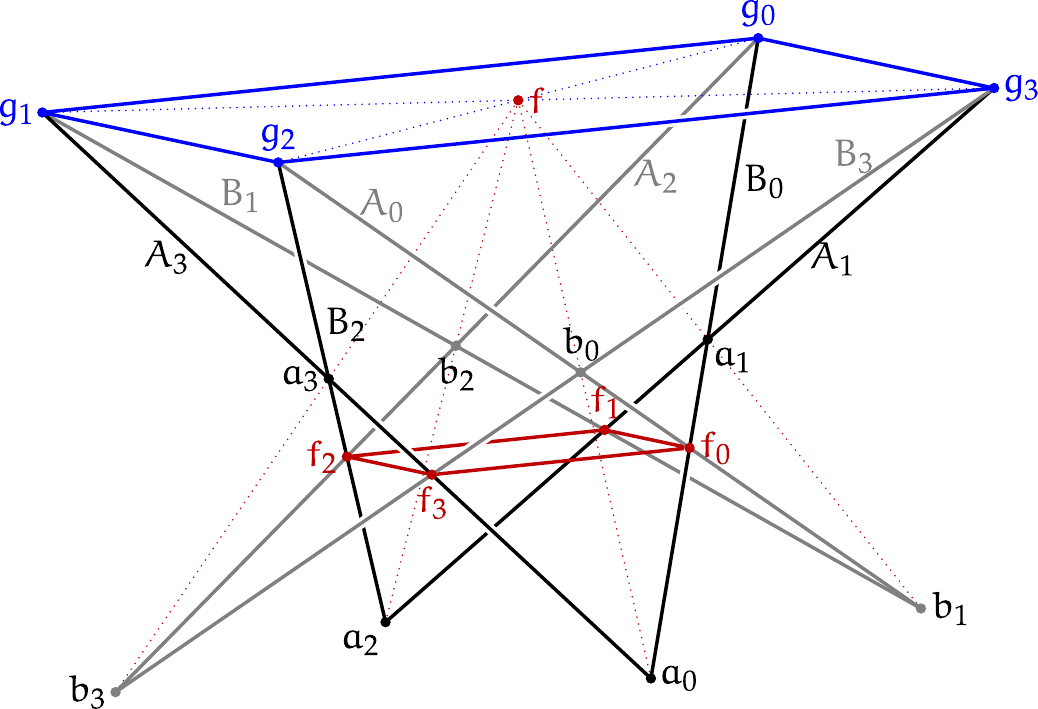}
  \caption{Illustration of Lemma~\ref{lem:18} and Lemma~\ref{lem:19}}
  \label{fig:point-plane}
\end{figure}

\begin{proof}
  The statement is self-dual. Thus, we only have to prove one
  implication. Assume that the four planes $\varphi_0$, $\varphi_1$,
  $\varphi_2$, $\varphi_3$ intersect in a point $f$ and consider the
  central perspectivity with center $f$ and axis $\overline{\varphi}
  \coloneqq f_0 \vee f_1 \vee f_2$ that transforms $a_0$ to
  $b_0$. Clearly, the image of all points $a_i$ is the corresponding
  point $b_i$ for $i \in \{0,1,2,3\}$. Thus, the point $f_3$ also lies
  in~$\overline{\varphi}$.
\end{proof}

\begin{lemma}
  \label{lem:19}
  Given are four skew lines $A_0$, $A_1$, $A_2$, $A_3$ of a regulus
  and a plane $\varphi$ not incident with either of these lines. For
  $i \in \{0,1,2,3\}$ set
  \begin{equation*}
    f_i \coloneqq \varphi \cap A_i
    \quad\text{and}\quad
    g_{i+2} \coloneqq A_{i+2} \cap (f_i \vee A_{i+1})
  \end{equation*}
  (indices modulo four). Then the four points $g_0$, $g_1$, $g_2$,
  $g_3$ lie in a plane as well (\autoref{fig:point-plane}).
\end{lemma}

\begin{proof}
  For $i \in \{0,1,2,3\}$ denote by $B_i$ the hyperboloid's second
  generator (different from $A_i$) through $f_i$ so that $g_{i+2} =
  A_{i+2} \cap B_i$ (indices modulo four). Consider now the two
  spatial quadrilaterals whose edges are the lines $A_0$, $B_1$,
  $A_2$, $B_3$ and $B_0$, $A_1$, $B_2$, $A_3$. Since their
  intersection points $f_0$, $f_1$, $f_2$, $f_3$ are coplanar, we can
  apply Lemma~\ref{lem:18} and see that the planes $\varphi_i
  \coloneqq A_i \vee B_i$ intersect in a point $f$. Moreover, the two
  spatial quadrilaterals correspond in a central perspectivity with
  center $f$ and plane of perspectivity $\varphi$. This implies that
  the point $f$ lies on $g_0 \vee g_2$ and $g_1 \vee g_3$. In
  particular, the four points $g_0$, $g_1$, $g_2$ and $g_3$ are
  coplanar.
\end{proof}

Theorem~\ref{th:17} gives us a simple means to test whether a
conjugate net $f$ has a Laplace sequence of period four. Admittedly,
an equally simple test consists of the construction of the Laplace
transforms. A more urgent characterization is that of the diagonal
congruences of a discrete Laplace cycle of period four. They are
necessarily W-congruences but this is not sufficient. Starting with a
W-congruence and an undermined face plane of an elementary
quadrilateral we can compute the intersection points of the face plane
with the corresponding axes and neighbouring vertices (by means of
considerations that already appear in our proof of
Theorem~\ref{th:13}).  Exploiting the planarity conditions, we arrive
at a system of algebraic equations that have no solutions in general.
Theorem~\ref{th:17} makes us conjecture the in special cases this
system can be reduced to a quadratic equation, thus giving rise to a
pair of two discrete conjugate nets on the given congruence. Maybe
even a continuum of solutions is feasible in non-trivial cases.

\subsection{Construction of Laplace cycles of period four}
\label{sec:construction}

Now we are going to present two methods for constructing a Laplace
cycle $f$, $h$, $g$, $k$ of period four. The first method works on the
level of the nets while the second method requires as input partial
information on the net $f$ and the diagonal congruence $A$
through~$f$.

\begin{figure}
  \centering
  \includegraphics{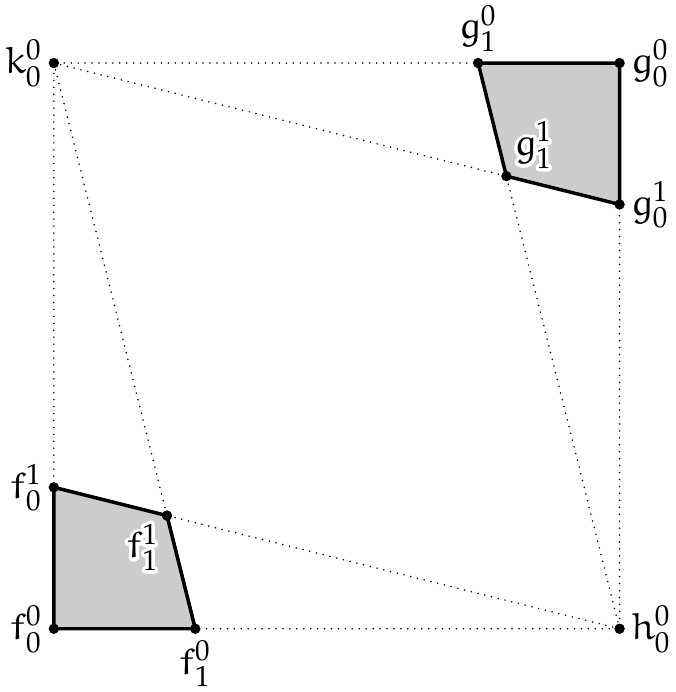}
  \quad
  \includegraphics{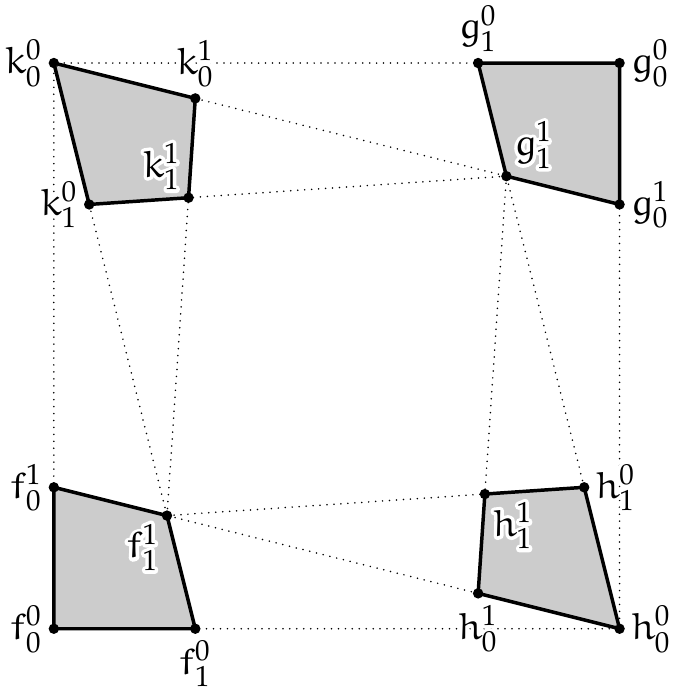}\\
  \includegraphics{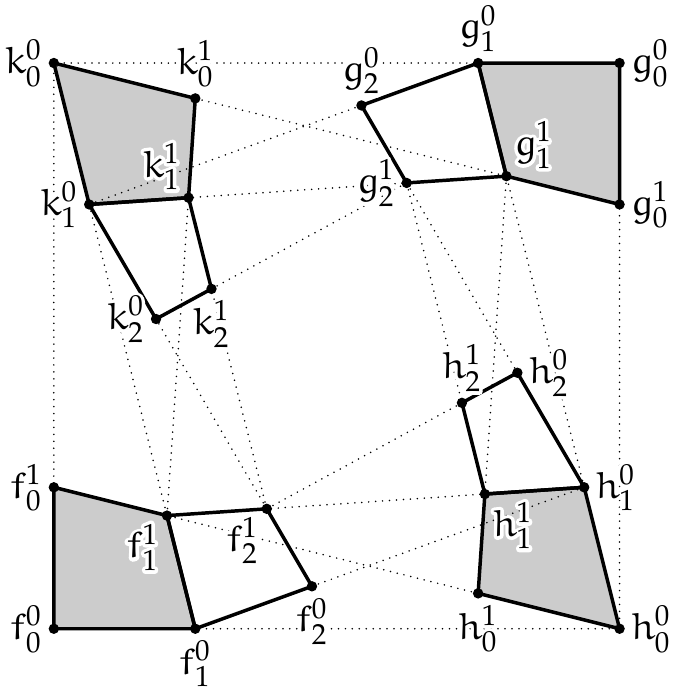}
  \quad
  \includegraphics{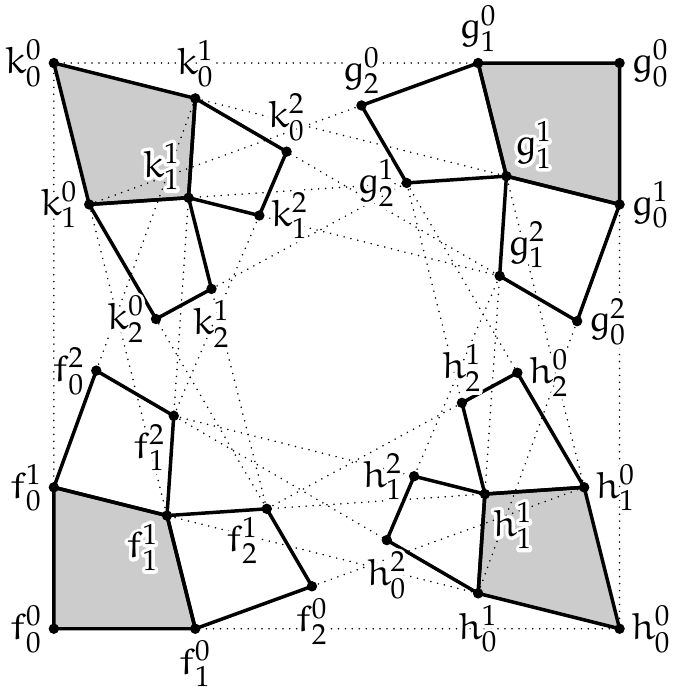}\\
  \includegraphics{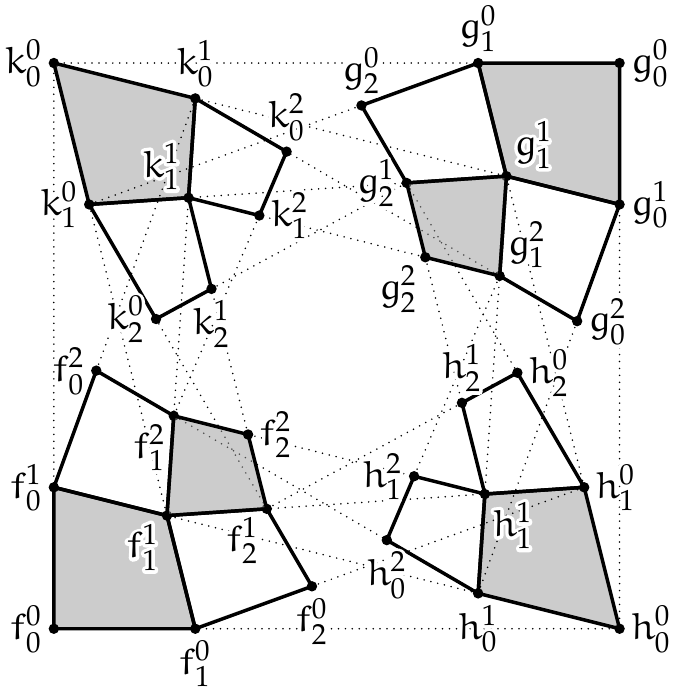}
  \quad
  \includegraphics{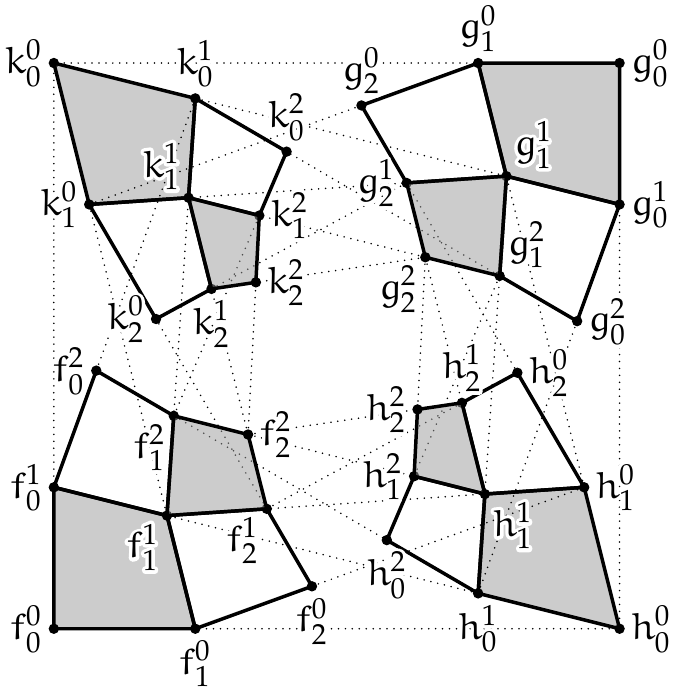}
  \caption{Construction of a Laplace chain of period four}
  \label{fig:laplace-construction}
\end{figure}

The first construction is illustrated in
\autoref{fig:laplace-construction}. The top-left picture shows the
initialization. We arbitrarily prescribe $f_0^0$, $g_0^0$, $h_0^0$,
and $k_0^0$. Then we have four degrees of freedom to choose
\begin{equation*}
  f_1^0 \in f_0^0 \vee h_0^0,\quad
  f_0^1 \in f_0^0 \vee k_0^0,\quad
  g_1^0 \in g_0^0 \vee k_0^0,\quad
  g_0^1 \in g_0^0 \vee h_0^0.
\end{equation*}
These points define
\begin{equation*}
  f_1^1 = (f_1^0 \vee k_0^0) \cap (f_0^1 \vee h_0^0)
  \quad\text{and}\quad
  g_1^1 = (g_1^0 \vee h_0^0) \cap (g_0^1 \vee k_0^0).
\end{equation*}
Next, we have four more degrees of freedom to choose points
\begin{equation*}
  h_1^0 \in h_0^0 \vee g_1^0,\quad
  h_0^1 \in h_0^0 \vee f_0^1,\quad
  k_1^0 \in k_0^0 \vee f_1^0,\quad
  k_0^1 \in k_0^0 \vee g_0^1
\end{equation*}
(\autoref{fig:laplace-construction}, top-right). They define
\begin{equation*}
  h_1^1 = (h_1^0 \vee f_1^1) \cap (h_0^1 \vee g_1^1)
  \quad\text{and}\quad
  k_1^1 = (k_1^0 \vee g_1^1) \cap (k_0^1 \vee f_1^1).
\end{equation*}

The same steps can be repeated for obtaining adjacent quadrilaterals
in the first or second net direction. The respective steps in the
first and second net direction are depicted in the middle row of
\autoref{fig:laplace-construction}. The only difference is that
certain input points are already prescribed so that only four degrees
of freedom per face remain. Once all points $f_i^j$, $g_i^j$, $h_i^j$,
$k_i^j$ with $i \in \{0,1\}$ and $j \in \ZSet$ or $j \in \{0,1\}$ and
$i \in \ZSet$ are found, the remaining points of all four nets are
uniquely determined (\autoref{fig:laplace-construction}, bottom
row). A Laplace cycle of period four in three-space is shown in
\autoref{fig:laplace-construction-3d}.

\begin{figure}
  \centering
  \begin{overpic}{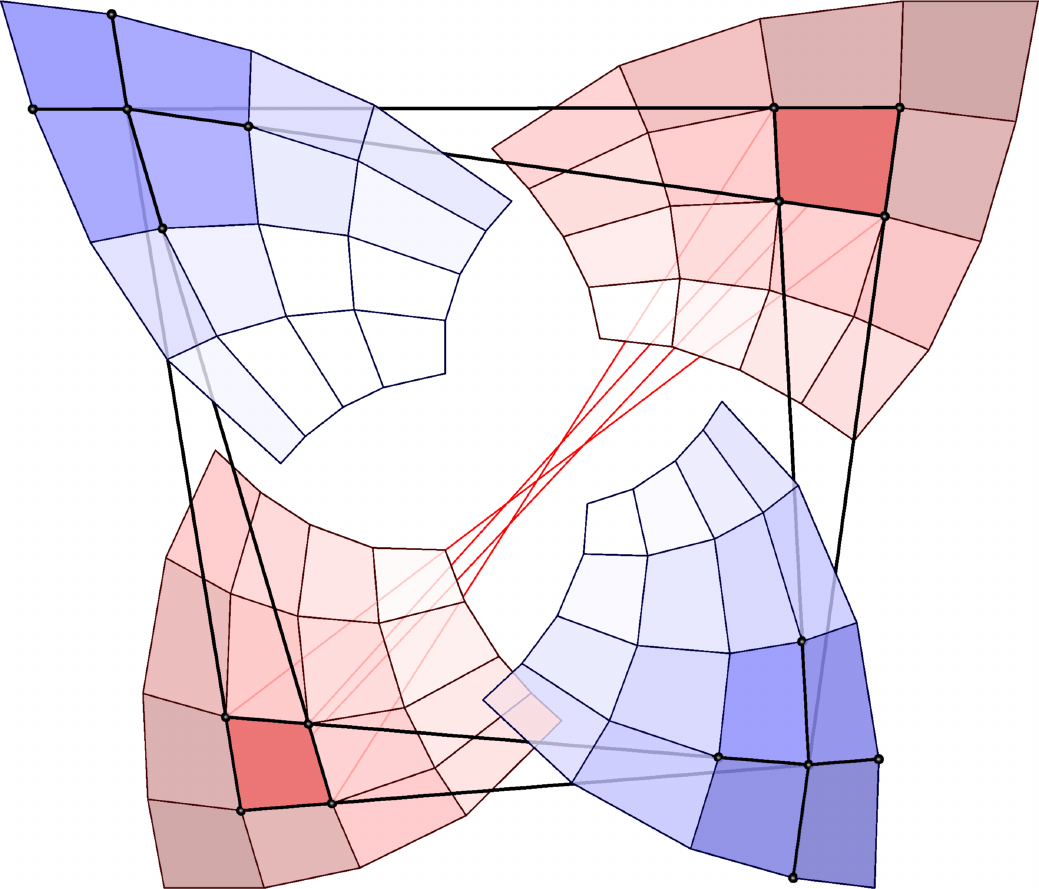}
    \put(20,5){\contour{white}{$f_i^j$}}
    \put(33,6){\contour{white}{$f_{i+1}^j$}}
    \put(17,17){\contour{white}{$f_i^{j+1}$}}
    \put(31,17){\contour{white}{$f_{i+1}^{j+1}$}}
    \put(87,77){\contour{white}{$g_i^j$}}
    \put(67,62){\contour{white}{$g_{i+1}^{j+1}$}}
    \put(86,62){\contour{white}{$g_i^{j+1}$}}
    \put(70,78){\contour{white}{$g_{i+1}^j$}}
    \put(78,8){\contour{white}{$h_i^j$}}
    \put(86,12){\contour{white}{$h_i^{j-1}$}}
    \put(63,13){\contour{white}{$h_i^{j+1}$}}
    \put(78,0){\contour{white}{$h_{i-1}^j$}}
    \put(78,25){\contour{white}{$h_{i+1}^j$}}
    \put(13,71){\contour{white}{$k_i^j$}}
    \put(2,70){\contour{white}{$k_i^{j+1}$}}
    \put(25,70){\contour{white}{$k_i^{j-1}$}}
    \put(17,63){\contour{white}{$k_{i+1}^j$}}
    \put(12,83){\contour{white}{$k_{i-1}^j$}}
  \end{overpic}
  \caption{A Laplace chain of period four}
  \label{fig:laplace-construction-3d}
\end{figure}
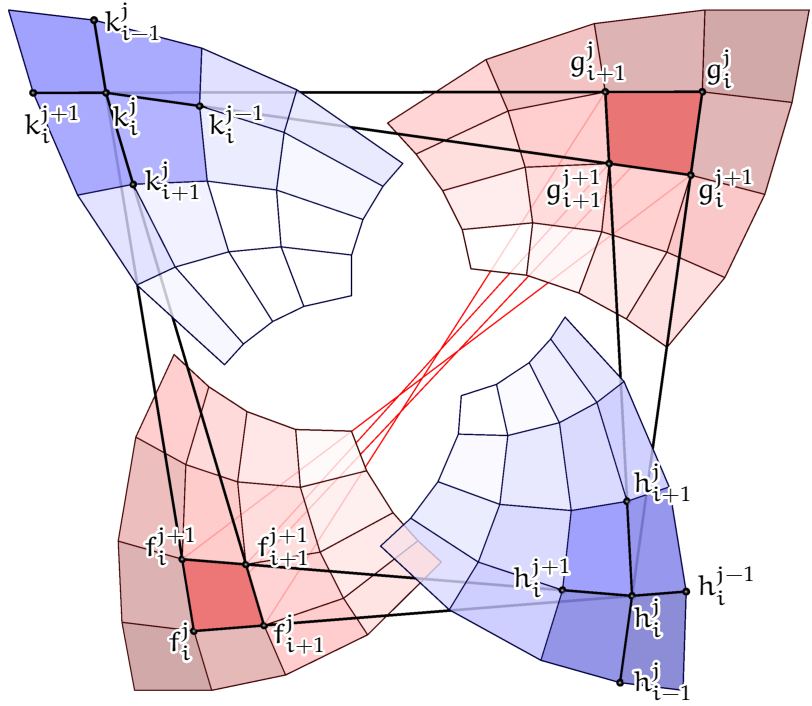

The presented construction simultaneously builds up the nets $f$, $h$,
$g$, and $k$. But it is also possible to directly determine only one
of the nets, say $f$, and its axis congruence $A$. The construction is
described in the proof of Theorem~\ref{th:19}, below. It requires some
auxiliary concepts and results.

\begin{definition}[projection between conics]
  Consider two conics $C$, $D \subset \Pspace$ with a common point $z$
  and a straight line $Z$ through $z$ that does not lie in the conics'
  planes. Denote by $\gamma$ the plane spanned by $Z$ and the tangent
  of $C$ in $z$ and by $\delta$ the plane spanned by $Z$ and the
  tangent of $D$ in $z$. If $\gamma \neq \delta$, set $c = \{C \cap
  \delta\} \setminus z$ and $d = \{D \cap \gamma\} \setminus
  z$. Otherwise, set $c = d = z$. The \emph{projection
    $\projection{Z}{C}{D}$ of $C$ onto $D$ from the center $Z$} is the
  map
  \begin{equation}
    \label{eq:9}
    x \in C \mapsto
    \begin{cases}
      z & \text{if $x = c$,} \\
      d & \text{if $x = z$,} \\
      \{D \cap (x \vee Z)\} \setminus z & \text{else.}
    \end{cases}
  \end{equation}
  This definition also makes sense if one or two of the conics
  consists of a pair of intersecting lines. In this case, any line in
  the conic's plane through the intersection point $s$ of the two lines is
  regarded as tangent. Moreover, we require $z \neq s$.
\end{definition}

We do not expect any confusion arising from the use of the same
notation for the projections between conics and between lines. In
fact, the projection between two lines $K$ and $L$ can be seen as a
degenerate case of the projection between the conics $K \cup M$ and $L
\cup M$ where $M$ is a transversal of $K$ and $L$. With this
understanding, we need not distinguish regular and degenerate conics
in the following.

The union of all points on the lines of a regulus is called a
\emph{quadric surface} or simply a \emph{quadric}
\cite[p.~301]{veblen16:_projective_geometry}. The intersection of a
quadric with a plane is a (possibly degenerate) conic.

\begin{lemma}
  \label{lem:21}
  Consider five lines $A_i^j$ and five points $f_i^j \in A_i^j$ with
  $(i,j) \in \{ (0,0), (\pm 1, 0), (0,\pm 1) \}$.  For $(i,j) = (\pm
  1, \pm 1)$ denote by $\quadric_i^j$ the quadric through $A_0^0$,
  $A_i^0$ and $A_0^j$ and by $\conic_i^j$ the (possibly degenerate)
  conic $\quadric_i^j \cap (f_0^0 \vee f_i^0 \vee f_0^j)$. Then the
  composition of projections
  \begin{equation*}
    \bigl(\projection{A_{-1}^0}{\conic_{-1}^1}{\conic_{-1}^{-1}}\bigr)
    \circ
    \bigl(\projection{A_0^1}{\conic_1^1}{\conic_{-1}^1}\bigr)
    \circ
    \bigl(\projection{A_1^0}{\conic_1^{-1}}{\conic_1^1}\bigr)
    \circ
    \bigl(\projection{A_0^{-1}}{\conic_{-1}^{-1}}{\conic_1^{-1}}\bigr)
  \end{equation*}
  is the identity on~$\conic_{-1}^{-1}$.
\end{lemma}

\begin{proof}
  Given a point $f_1^1 \in \conic_1^1$, there exists a unique line
  $B_1^1$ on $\quadric_1^1$ that intersects $A_0^0$ in a point
  $b_1^1$. Conversely, any points $b_1^1 \in A_0^0$ gives rise to a
  point $f_1^1 \in \conic_1^1$. Similarly, the points of
  $\conic_{-1}^1$, $\conic_1^{-1}$, and $\conic_{-1}^{-1}$ are related
  to the points of $A_0^0$. The crucial observation is now that
  $f_{-1}^1$ and $f_1^1$ correspond in the projection
  $\projection{A_0^1}{\conic_{-1}^1}{\conic_1^1}$ if and only if their
  corresponding points on $A_0^0$ are equal. This is true due to
  \begin{equation*}
    b_{-1}^1 = (f_{-1}^1 \vee A_0^1) \cap A_0^0 = (f_1^1 \vee A_0^1) \cap A_0^0 = b_1^1.
  \end{equation*}
  Thus, starting from $f_1^1$, we can construct the remaining points
  $f_{-1}^1$, $f_{-1}^{-1}$, and $f_1^{-1}$ and they all correspond to
  the \emph{same} point $b_1^1 = b_{-1}^1 = b_{-1}^{-1} = b_1^{-1} \in
  A_0^0$. We infer that after four successive projections, we obtain
  the initial point $f_1^1$. This finishes the proof.
\end{proof}

\begin{theorem}
  \label{th:19}
  Consider a Laplace cycle $f$, $h$, $g$, $k$, of period four and
  denote its diagonal congruence by $A$. The Laplace cycle is uniquely
  determined by
  \begin{itemize}
  \item suitable values of $f$ and $A$ on all vertices $(i,0)$ and
    $(0,i)$ with $i \in \ZSet$ and
  \item the value of $f_1^1$ in the intersection of the plane $f_0^0
    \vee f_1^0 \vee f_0^1$ and the quadric spanned by the three lines
    $A_0^0$, $A_1^0$, and $A_0^1$.
  \end{itemize}
  Here, the points $f_i^0$, $f_0^i$ and lines $A_i^0$, $A_0^i$ are
  said to be in ``suitable position'', if for any $i \in \ZSet$ the
  conditions
  \begin{itemize}
  \item $A_i^0 \cap A_{i+1}^0 = \varnothing$, $A_0^i \cap A_0^{i+1} =
    \varnothing$,
  \item $A_i^0 \subset f_{i-1}^0 \vee f_i^0 \vee f_{i+1}^0$,
    $A_0^i \subset f_0^{i-1} \vee f_0^i \vee f_0^{i+1}$, and
  \item $f_i^0 \in A_i^0$, $f_0^i \in A_0^i$
  \end{itemize}
  are fulfilled.
\end{theorem}

\begin{proof}
  We observe at first, that the given data determines the nets $f$ and
  $A$ uniquely. Since $f_1^1$ lies on the quadric of $A_0^0$, $A_1^0$,
  and $A_0^1$, the line $A_1^1$ is well-defined.  The vertex
  $f_{-1}^1$ necessarily lies in the planes
  \begin{equation*}
    f_1^1 \vee A_0^1
    \quad\text{and}\quad
    f_0^0 \vee f_0^1 \vee f_{-1}^0
  \end{equation*}
  and on the quadric $\quadric$ through $A_0^0$, $A_0^1$,
  $A_{-1}^0$. Thus, it is the projection of $f_1^1$ from the conic
  $\conic_1^1 = \quadric_1^1 \cap (f_0^0 \vee f_1^0 \vee f_0^1)$ onto
  the conic $\conic_{-1}^1 = \quadric_{-1}^1 \cap (f_0^0 \vee f_{-1}^0
  \vee f_1^1)$.  Once $f_{-1}^1$ is found, the line $A_{-1}^1$ is
  determined as well. Proceeding in like manner, we can inductively
  construct $f$ as a discrete conjugate net and $A$ as a discrete
  W-congruence. This shows uniqueness of $f$ and $A$ and, by
  Theorem~\ref{th:17}, also of the complete Laplace cycle.

  As to existence, we have to consider all possibilities to run into a
  contradiction. For example, we have to guarantee that the values for
  constructing $f_{-1}^{-1}$ via the routes
  \begin{equation*}
    f_1^1 \xrightarrow{\projection{A_0^1}{\conic_1^1}{\conic_{-1}^1}} f_{-1}^1 \xrightarrow{\projection{A_{-1}^0}{\conic_{-1}^1}{\conic_{-1}^{-1}}} f_{-1}^{-1}
    \quad\text{or}\quad
    f_1^1 \xrightarrow{\projection{A_1^0}{\conic_1^1}{\conic_1^{-1}}} f_1^{-1} \xrightarrow{\projection{A_0^{-1}}{\conic_1^{-1}}{\conic_{-1}^{-1}}} f_{-1}^{-1}
  \end{equation*}
  yield the same result. This is indeed the case by virtue of
  Lemma~\ref{lem:21}. This lemma also shows that, given the vertices
  $f_i^j$ for all $i$, $j \in \{-1, 0, 1\}$ plus the vertex $f_2^0$,
  the vertices $f_2^{-1}$ and $f_2^1$ can be constructed without
  contradiction. The same is true for all vertices $f_i^1$, $f_i^{-1}$
  by induction and for all vertices $f_1^j$, $f_{-1}^j$ by
  analogy. Finally, Lemma~\ref{lem:21} is also responsible for the
  fact that given the vertices $f_i^j$ for $i$, $j \in \{0,1,2\}$ but
  different from $i=j=2$, the two possibilities for constructing
  $f_2^2$ coincide. This allows the contradiction-free construction of
  all remaining vertices so that existence of the conjugate net $f$ is
  shown as well.
\end{proof}

\section{Conclusion}
\label{sec:conclusion}

We have shown that many results of \cite{jonas37:_laplacesche_zyklen}
on asymptotically related nets and, in particular, conjugate nets also
hold in a discrete setting, thus adding some new insight into the
already established geometry of discrete conjugate and asymptotic nets
and their transformations. It turned out that some of Jonas' results
admit similar but not identical discrete interpretations, for example,
Theorem~\ref{th:13} and Corollary~\ref{cor:14}.

The subject of our research is not yet exhausted. Besides the
characterization of diagonal congruences of Laplace cycles of period
four mentioned in \autoref{sec:general-results}, a worthy topic of
future research seem to be the considerations of
\cite{jonas37:_allgemeine_transformationstheorie} on transformations
of Laplace cycles of period four. A proper discrete transformation
theory could integrate our topic into the ``consistency as
integrability'' paradigm of
\cite{bobenko08:_discrete_differential_geometry}. We will attempt this
in a future publication.




\par\bigskip
\flushright
\begin{minipage}{0.55\linewidth}
\footnotesize
Hans-Peter Schröcker, Unit Geometry and CAD, University Innsbruck,
Technikerstraße 13, 6020 Innsbruck, Austria.\\
Email: \href{mailto:hans-peter.schroecker@uibk.ac.at}{\nolinkurl{hans-peter.schroecker@uibk.ac.at}}\\
URL: \url{http://geometrie.uibk.ac.at/schroecker/}
\end{minipage}

\end{document}